\documentclass[11pt]{article}
\usepackage{latexsym, amsfonts, amssymb, amsthm, amsmath, hyperref, graphicx, wrapfig}
\usepackage[mathscr]{eucal}
\usepackage[usenames,dvipsnames]{xcolor}

\newtheorem{thm}{Theorem}[section]
\newtheorem{prop}[thm]{Proposition}
\newtheorem{cor}[thm]{Corollary}

\newtheorem{question}[thm]{Question}

\newtheorem*{chickenthm}{The Chicken McNugget Problem}
\newtheorem*{nummonthm}{The Numerical Monoid Problem}
\theoremstyle{definition}
\newtheorem{definition}[thm]{Definition}

\newtheorem{example}[thm]{Example}

\newcommand{\nm}{\langle n_1, \ldots, n_k\rangle}
\newcommand{\cmm}{\ensuremath{%
	\mathchoice
		{\includegraphics[height=1.7ex]{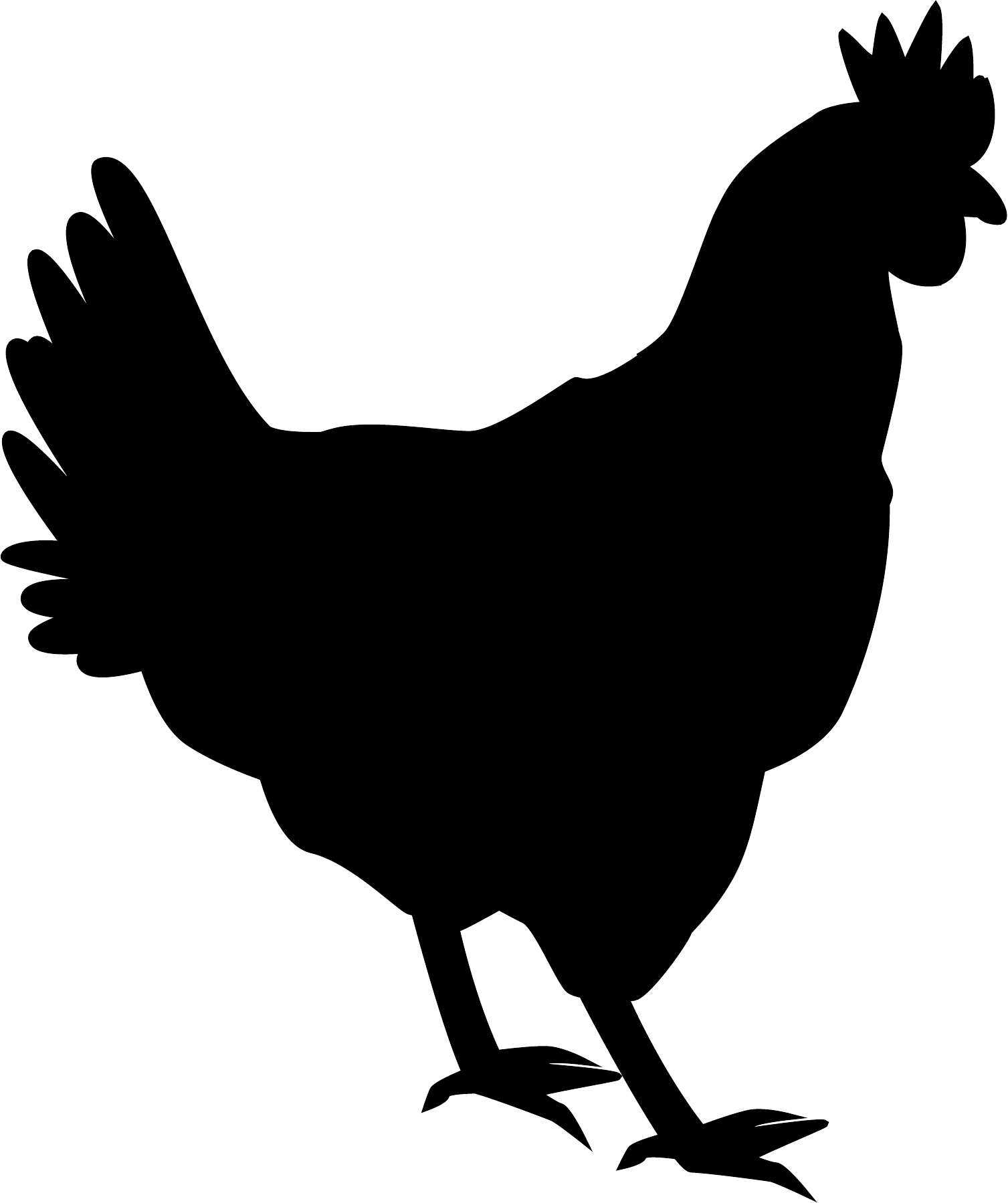}}
		{\includegraphics[height=1.7ex]{hen-311285.pdf}}
		{\includegraphics[height=1.5ex]{hen-311285.pdf}}
		{\includegraphics[height=1.0ex]{hen-311285.pdf}}
}}

\newcommand{\frob}{F}

\begin{document}
\begin{center}
\Large
Factoring in the Chicken McNugget monoid\footnote{This article is based on a 2013 PURE Mathematics REU Project by Emelie Curl, Staci Gleen, and Katrina Quinata which was directed by the authors and Roberto Pelayo.}
\end{center}
\medskip


\begin{flushright}
Scott T. Chapman \\
Sam Houston State University \\
Department of Mathematics and Statistics\\
Box 2206\\
Huntsville, TX  77341-2206\\
\verb+scott.chapman@shsu.edu+
\end{flushright}
\bigskip

\begin{flushright}
Christopher O'Neill \\
University of California at Davis \\
Department of Mathematics\\
One Shield Ave.\\
Davis, CA  95616\\
\verb+coneill@math.ucdavis.edu+
\end{flushright}
\bigskip

\begin{center}
\textit{People just want more of it.} - Ray Kroc \cite{BQ}
\end{center}

\medskip

\section{Introduction}
\label{sec:intro}

Every day, 34 million Chicken McNuggets are sold worldwide \cite{AN}.  
At most McDonalds locations in the United States today, Chicken McNuggets are sold in packs of 4, 6, 10, 20, 40, and 50 pieces.  However, shortly after their introduction in 1979 they were sold in packs of 6, 9, and 20.  The following problem spawned from the use of these latter three numbers.

\begin{chickenthm}
What numbers of Chicken McNuggets can be ordered using only packs with 6, 9, or 20 pieces?
\end{chickenthm}


Early references to this problem can be found in \cite{PW,V}.  Positive integers satisfying the Chicken McNugget Problem are now known as \textit{McNugget numbers}~\cite{McN}.  In particular, if $n$ is a McNugget number, then there is an ordered triple $(a,b,c)$ of nonnegative integers such that
\begin{equation}\label{basiceq}
6a + 9b + 20c = n.
\end{equation}
We will call $(a,b,c)$ a \textit{McNugget expansion} of $n$ (again see \cite{McN}).  Since both $(3,0,0)$ and $(0,2,0)$ are McNugget expansions of 18, it is clear that McNugget expansions are not unique.  This phenomenon will be the central focus of the remainder of this article.

If $\max\{a,b,c\} \geq 8$ in \eqref{basiceq}, then $n \geq 48$ and hence determining the numbers $x$ with $0 \leq x \leq 48$ that are McNugget numbers can be checked either by hand or your favorite computer algebra system.  The only such $x$'s that are not McNugget numbers are: 1, 2, 3, 4, 5, 7, 8, 10, 11, 13, 14, 16, 17, 19, 22, 23, 25, 28, 31, 34, 37, and 43.  (The non-McNugget numbers are sequence A065003 in the On-Line Encyclopedia of Integer Sequences \cite{OEIS}.)  We~demonstrate this in Table~\ref{tb:mcnuggetexpansions} with a chart that offers the McNugget expansions (when they exist) of all numbers $\leq 50$.

\begin{table}[t]\label{tb:mcnuggetexpansions}
\small{
\[
\begin{array}{||c|c||c|c||c|c||}
\hline
\# & (a,b,c) & \# & (a,b,c) & \# & (a,b,c) \\
\hline
\mathbf{0} & (0,0,0) & 
17 & \mbox{NONE} & 
34 & \mbox{NONE} \\
\hline
1 & \mbox{NONE} & 
\mathbf{18} & (3,0,0) \, (0,2,0) & 
\mathbf{35} & (1,1,1) \\
\hline
2 & \mbox{NONE} & 
19 & \mbox{NONE} & 
\mathbf{36} & (0,4,0) \, (3,2,0) \, (6,0,0) \\
\hline 
3 & \mbox{NONE} & 
\mathbf{20} & (0,0,1) & 
37 & \mbox{NONE} \\
\hline 
4 & \mbox{NONE} & 
\mathbf{21} & (2,1,0) & 
\mathbf{38} & (0,2,1) \, (3,0,1) \\
\hline 
5 & \mbox{NONE} & 
22 & \mbox{NONE} & 
\mathbf{39} & (2,3,0) \, (5,1,0) \\
\hline
\mathbf{6} & (1,0,0) & 
23 & \mbox{NONE} & 
\mathbf{40} & (0,0,2) \\
\hline 
7 & \mbox{NONE} & 
\mathbf{24} & (4,0,0) \, (1,2,0) & 
\mathbf{41} & (2,1,1) \\
\hline
8 & \mbox{NONE} & 
25 & \mbox{NONE} & 
\mathbf{42} & (1,4,0) \, (4,2,0) \, (7,0,0) \\
\hline
\mathbf{9} & (0,1,0) & 
\mathbf{26} & (1,0,1) & 
43 & \mbox{NONE}\\
\hline
10 & \mbox{NONE} & 
\mathbf{27} & (0,3,0) \, (3,1,0) & 
\mathbf{44} & (1,2,1) \, (4,0,1) \\
\hline
11 & \mbox{NONE} & 
28 & \mbox{NONE} & 
\mathbf{45} & (0,5,0) \, (3,3,0) \, (6,1,0) \\
\hline
\mathbf{12} & (2,0,0) & 
\mathbf{29} & (0,1,1) & 
\mathbf{46} & (1,0,2) \\
\hline
13 & \mbox{NONE} & 
\mathbf{30} & (5,0,0) \, (2,2,0) & 
\mathbf{47} & (0,3,1) \, (3,1,1) \\
\hline
14 & \mbox{NONE} & 
31 & \mbox{NONE} & 
\mathbf{48} & (2,4,0) \, (5,2,0) \, (8,0,0) \\
\hline
\mathbf{15} & (1,1,0) & 
\mathbf{32} & (2,0,1) & 
\mathbf{49} & (0,1,2) \\
\hline
16 & \mbox{NONE} & 
\mathbf{33} & (1,3,0) \, (4,1,0) & 
\mathbf{50} & (2,2,1) \, (5,0,1) \\
\hline
\end{array}
\]}
\caption{The McNugget numbers and their expansions from 0 to 50.
}
\end{table}

What happens with larger values?  Table~\ref{tb:mcnuggetexpansions} has already verified that 44, 45, 46, 47, 48, and 49 are McNugget numbers.  Hence, we have a sequence of 6 consecutive McNugget numbers,
and by repeatedly adding 6 to these values, we obtain the following.

\begin{prop}
Any $x > 43$ is a McNugget number.
\end{prop}

\noindent
Thus, 43 is the largest number of McNuggets that cannot be ordered with packs of 6, 9, and 20.  

Our aim in this paper is to consider issues related to the multiple occurances of McNugget expansions as seen in Table~\ref{tb:mcnuggetexpansions}.  Such investigations fall under the more general purview of the theory of non-unique factorizations in integral domains and monoids (a good technical reference on this subject is~\cite{GHKb}).  Using a general context, we show that the McNugget numbers form an additive monoid and discuss some properites shared by the class of additive submonoids of the nonnegative integers.  We then define several combinatorial characteristics arising in non-unique factorization theory, and compute their explicit values for the McNugget Monoid.  

\begin{wrapfigure}{l}{0\linewidth}
\includegraphics[width=2in]{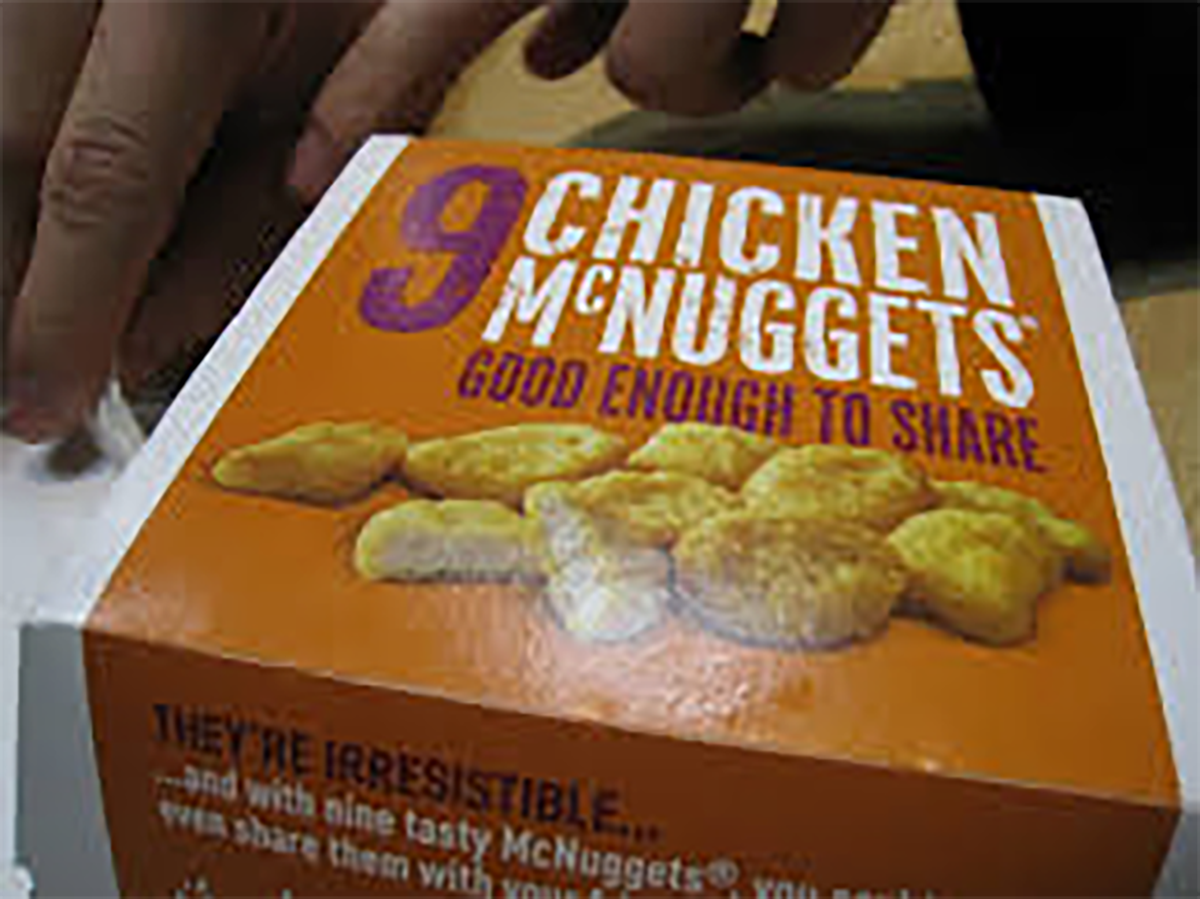}
\caption{The 9 piece box.}
\end{wrapfigure}

By emphasizing results concerning McNugget numbers, we offer the reader a glimpse into the vast literature surrounding non-unque factorizations.  While we stick to the calculation of basic factorization invariants, our results indicate that such computations involve a fair amount of complexity.  Many of the results we touch on have appeared in papers authored or co-authored by undergraduates
in National Science Foundation Sponsored REU Programs.  This is an area that remains rich in open problems, and we hope our discussion here spurs our readers (both young and old) to explore this rewarding subject more deeply.

\section{A brief diversion into generality}
\label{sec:generality}

As illustrated above, Chicken McNugget numbers fit into a long studied mathematical concept.  Whether called the Postage Stamp Problem \cite{PSP}, the Coin Problem \cite{CP}, or the Knapsack Problem \cite{KP}, the idea is as follows.  Given a set of $k$ objects with predetermined values $n_1, n_2, \ldots , n_k$, what possible values of $n$ can be had from combinations of these objects?  Thus, if a value of $n$ can be obtained, then there is an ordered $k$-tuple of nonnegative integers $(x_1, \ldots, x_k)$ that satisfies the linear diophatine equation
\begin{equation}\label{anothereq}
n = x_1n_1 + x_2n_2 + \cdots + x_kn_k.
\end{equation}
We view this in a more algebraic manner.  Given integers $n_1, \ldots, n_k > 0$, set
\[
\langle n_1, \ldots, n_k \rangle = \{x_1n_1 + \cdots + x_kn_k \mid x_1, \ldots, x_k \in \mathbb N_0\}.
\]
Notice that if $s_1$ and $s_2$ are in $\nm$, then $s_1 + s_2$ is also in $\nm$.  
Since $0 \in \nm$ and $+$ is an associative operation, the set $\nm$ under $+$ forms a \textit{monoid}.  
Monoids of nonnegative integers under addition, like the one above, are known as \textit{numerical monoids}, and $n_1,\ldots ,n_k$ are called \textit{generators}.  We will call the numerical monoid $\langle 6, 9, 20 \rangle$ the \textit{Chicken McNugget monoid}, and denote it by $\cmm$.  

\begin{wrapfigure}{l}{0\linewidth}
\includegraphics[width=2in]{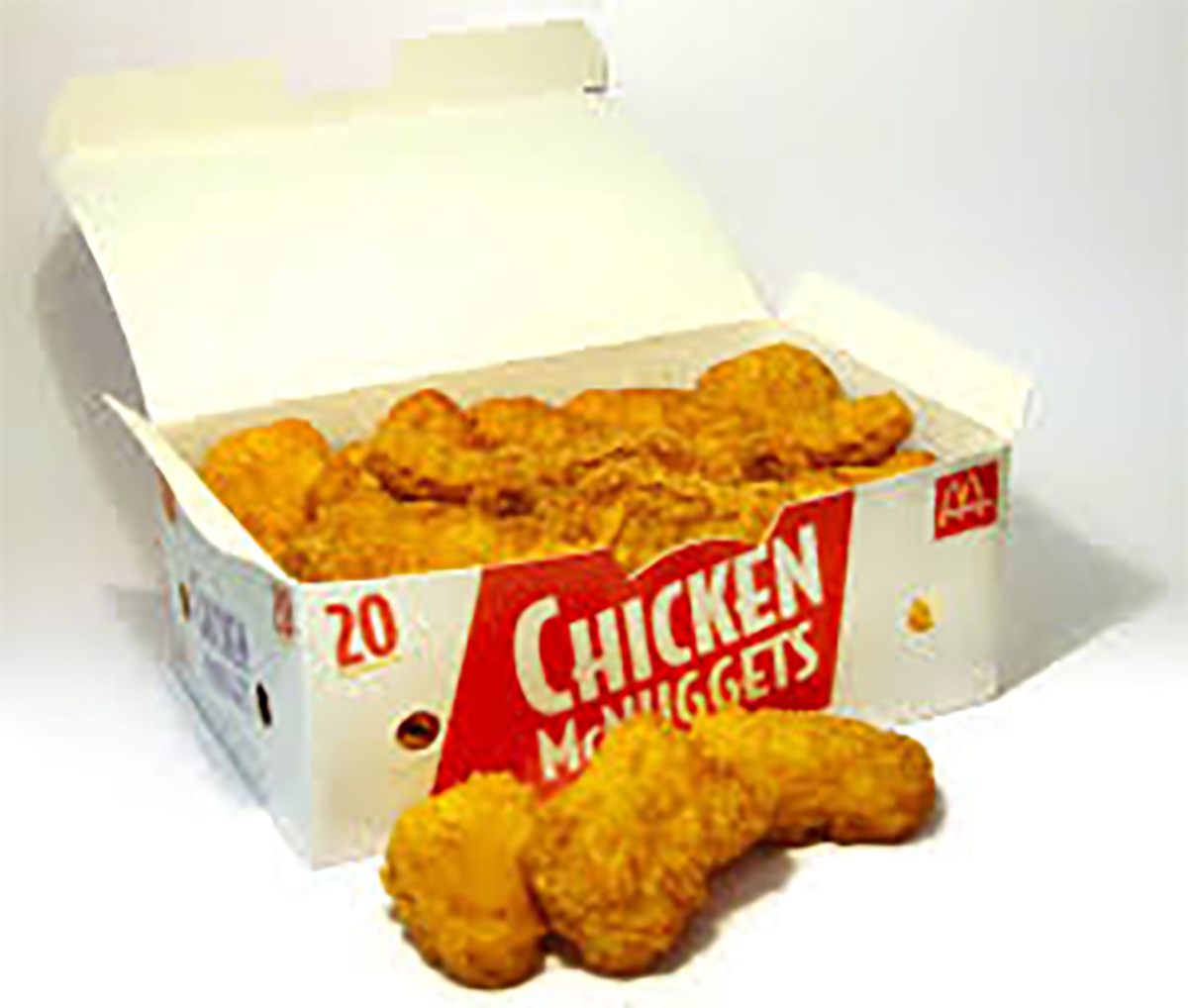}
\caption{The 20 piece box}
\end{wrapfigure}

Since $\cmm$ consists of the same elements as those in $\langle 6, 9, 20, 27\rangle$, it is clear that generating sets are not unique.  Using elementary number theory, it is easy to argue that any numerical monoid $\nm$ does have a unique generating set with minimal cardinality obtained by eliminating those generators $n_i$ that lie in the numerical monoid generated by $\{n_1, \ldots, n_k\} - \{n_i\}$.  In this way, it is clear that $\{6, 9, 20\}$ is indeed the minimal generating set of~$\cmm$.  When dealing with a general numerical monoid $\nm$, we will assume without loss of generality that the given generating set $\{n_1, \ldots, n_k\}$ is minimal.

In view of this broader setting, the Chicken McNugget Problem can be generalized as follows.

\begin{nummonthm}
If $n_1, \ldots, n_k$ are positive integers, then which nonnegative integers lie in $\nm$?
\end{nummonthm}

\begin{example}\label{ex1}
We have already determined above exactly which nonnegative integers are McNugget numbers.  Suppose the Post Office issues stamps in denominations of 4 cents, 7 cents, and 10 cents.  What values of postage
can be placed on a letter (assuming that as many stamps as necessary can be placed on the envelope)?  In particular, we are looking for the elements of $\langle 4, 7, 10 \rangle$.  We can again use brute force to find all the solutions to
\[
4a + 7b + 10c = n
\]
and conclude that 1, 2, 3, 5, 6, 9, and 13 cannot be obtained.  Since 14, 15, 16, and 17 can, all postage values larger than 13 are possible. \hfill $\Box$
\end{example}


Let's return to the largest number of McNuggets that can't be ordered (namely, 43) and the companion number 13 obtained in Example~\ref{ex1}.  The existence of these numbers is no accident.  To see this in general, let $n_1,\ldots, n_k$ be a set of positive integers that are relatively prime.  By elementary number theory, there is a set $y_1, \ldots, y_k$ of (possibly negative) integers such that 
\[
1 = y_1n_1 + \cdots + y_kn_k.
\]
By choosing an element $V = x_1n_1 + \cdots + x_kn_k \in \nm$ with sufficiently large coefficients (for instance, if each $x_i \ge n_1|y_i|$), we see $V + 1, \ldots, V + n_1$ all lie in $\nm$ as well.  As such, any integer greater than $V$ can be obtained in $\nm$ by adding copies of $n_1$.  

This motivates the following definition.  

\begin{definition}
If $n_1, \ldots, n_k$ are relatively prime positive integers, then the \textit{Frobenius number} of $\nm$, denoted $\frob(\nm)$, is the largest positive integer $n$ such that $n \notin \nm$.
\end{definition}



We have already shown that $\frob(\cmm) = 43$ and $\frob(\langle 4, 7, 10 \rangle) = 13$.  A famous result of Sylvester from 1884 \cite{Syl} states that if $a$ and $b$ are relatively prime, then $\frob(\langle a, b \rangle) = ab - a - b$ (a nice proof of this can be found in \cite{Be}).  This is where the fun begins, as strictly speaking no closed formula exists for the Frobenius number of numerical monoids that require 3 or more generators.  While there are fast algorithms that can compute $\frob(\langle n_1, n_2, n_3 \rangle)$ (see for instance \cite{ELSW}), at best formulas for $\frob(\nm)$ exist only in special cases (you can find one such special case where $\frob(\cmm) = 43$ pops out in \cite[p.~14]{AC}).  Our purpose is not to compile or expand upon the vast literature behind the Frobenius number; in fact, we direct the reader to the excellent monograph of Ram\'irez Alfons\'in \cite{RA} for more background reading on the Diophatine Frobenius Problem.


\section{The McNugget factorization toolkit}
\label{sec:toolkit}

We focus now on the multiple McNugget expansions we saw in Table~\ref{tb:mcnuggetexpansions}. In particular, notice that there are McNugget numbers that have unique triples associated to them (6, 9, 12, 15, 20, 21, 26, 29, 32, 35, 40, 41, 46, and 49), some that have two (18, 24, 27, 30, 33, 35, 39, 44, 47, and 50), and even some that have three (36, 42, 45, and 48).  While the ``normal'' notion of factoring occurs in systems where multiplication prevails, notice that the ordered triples representing McNugget numbers are actually \textit{factorizations} of these numbers into ``additive'' factors of 6, 9, and 20.

Let's borrow some terminology from abstract algebra (\cite{Ga} is a good beginning reference on the topic).  Let $x$ and $y\in \nm$.  We say that $x$ \textit{divides} $y$ in $\nm$ if there exists a $z\in \nm$ such that $y=x+z$.  We call a nonzero element $x \in \nm$ \textit{irreducible} if whenever $x = y + z$, either $y = 0$ or $z = 0$.  (Hence, $x$ is irreducible if its only proper divisors are 0 and itself).  Both of these definitions are obtained from the usual ``multiplicative'' definition by replacing ``$\cdot$'' with ``$+$'' and 1 with 0.  

We leave the proof of the following to the reader.

\begin{prop}\label{p:generalatoms}
If $\nm$ is a numerical monoid, then its irreducible elements are precisely $n_1,\ldots ,n_k$. 
\end{prop}

Related to irreducibility is the notion of prime elements.  A nonzero element $x \in \nm$ is \textit{prime} if whenever $x$ divides a sum $y+z$, then either $x$ divides $y$ or $x$ divides $z$ (this definition is again borrowed from the multiplicative setting).  It is easy to check from the definitions that prime elements are always irreducible, but it turns out that in general irreducible elements need not be prime.  In fact, the irreducible elements $n_1, \ldots, n_k$ of~a numerical monoid are never prime.   To see this, let $n_i$ be an irreducible element and let $T$ be the numerical monoid generated by $\{n_1, \ldots, n_k\} - \{n_i\}$.  Although $n_i \notin T$, some multiple of $n_i$ must lie in~$T$ (take, for instance, $n_2n_i$).  Let $kn = \sum_{j\neq i}x_jn_j$ (for some $k > 1$) be the smallest multiple of $n_i$ in~$T$.  Then $n$ divides $\sum_{j \neq i} x_jn_j$ over $\nm$, but by the minimality of $k$, $n$~does not divide any proper subsum.  Thus $n_i$ is not prime.

For our purposes, we restate Proposition~\ref{p:generalatoms} in terms of $\cmm$.

\begin{cor}\label{c:mcnuggetatoms}
The irreducible elements of the McNugget monoid are 6, 9, and 20.  There are no prime elements.
\end{cor}

\subsection{The set of factorizations of an element}
\label{sec:setoffactorizations}

We refer once again to the elements in Table~\ref{tb:mcnuggetexpansions} with multiple irreducible factorizations.  For each $x \in \cmm$, let
\[
\mathsf Z(x)=\{(a,b,c)\,|\, 6a+9b+20c = x\}.
\]
We will refer to $\mathsf Z(x)$ as the \textit{complete set of factorizations} $x$ in $\cmm$, and as such, we could relabel columns 2, 4, and 6 of Table~\ref{tb:mcnuggetexpansions} as ``$\mathsf Z(x)$.''  While we will not dwell on general structure problems involving $\mathsf Z(x)$, we do briefly address one in the next example.

\begin{example}\label{chickenunique}
What elements $x$ in the McNugget monoid are uniquely factorable (i.e., $|\mathsf Z(x)| = 1$)?  A quick glance at Table~\ref{tb:mcnuggetexpansions} yields 14 such nonzero elements (namely, 6, 9, 12, 15, 20, 21, 26, 29, 32, 35, 40, 41, 46, 49).  Are~there others?  We begin by noting in Table~\ref{tb:mcnuggetexpansions} that 
\[
(3,0,0), (0,2,0) \in \mathsf Z(18) \qquad \mbox{and} \qquad (10,0,0),(0,0,3) \in \mathsf Z(60).
\]
This implies that in any factorization in $\cmm$, 3~copies of~6 can be freely replaced with 2~copies of~9 (this is called a \textit{trade}).  Similarly, 2~copies of~9 can be traded for 3~copies of~6, and 3~copies of~20 can be traded for 10~copies of~6.  In particular, for $n = 6a + 9b + 20c \in \cmm$, if either $a \geq 3$, $b \geq 2$ or $c \geq 3$, then $n$ has more than one factorization in $\cmm$.  As such, if $n$ is to have unique factorization, then $0 \leq a \leq 2$, $0 \leq b \leq 1$, and $0 \leq c \leq 2$.  This leaves 18 possibilities, and a quick check yields that the 3 missing elements are $52 = (2,0,2)$, $55 = (1,1,2)$ and $61 = (2,1,2)$. \hfill $\Box$.  

The argument in Example \ref{chickenunique} easily generalizes -- every numerical monoid that requires more than one generator has finitely many elements that factor uniquely -- but note that minimal trades need not be as simple as replacing a multiple of one generator with a multiple of another.  Indeed, in the numerical monoid $\langle 5, 7, 9, 11 \rangle$, there is a trade $(1,0,0,1), (0,1,1,0) \in \mathsf Z(16)$, though 16 is not a multiple of any generator.  Determining the ``minimal'' trades of a numerical monoid, even computationally, is known to be a very hard problem in general \cite{Stu}.  
\end{example}

\subsection{The length set of an element and related invariants}
\label{sec:lengthset}

Extracting information from the factorizations of numerical monoid elements (or even simply writing them all down) can be a tall order.  To this end, combinatorially-flavored \emph{factorization invariants} are often used, assigning to each element (or to the monoid as a whole) a value measuring its failure to admit unique factorization.  We devote the remainder of this paper to examining several factorization invariants, and what they tell us about the McNugget monoid as compared to more general numerical monoids.  

We begin by considering a set, derived from the set of factorizations, that has been the focus of many papers in the mathematical literature over
the past 30 years.  If $x \in \cmm$ and $(a,b,c) \in \mathsf Z(x)$, then the \textit{length} of the factorization $(a,b,c)$ is denoted by
\[
|(a,b,c)| = a + b + c.
\]
We have shown earlier that factorizations in $\cmm$ may not be unique, and a quick look at Table~\ref{tb:mcnuggetexpansions} shows that their lengths can also differ.  For instance, 42 has three different factorizations, with lengths 5, 6 and 7, respectively.  Thus, we denote the \textit{set of lengths} of $x$ in $\cmm$ by
\[
\mathcal L(x) = \{|(a,b,c)| : (a,b,c) \in \mathsf Z(x)\}.
\]
In particular, $\mathcal L(42) = \{5,6,7\}$.   Moreover, set 
\[
\ell(x) = \min\mathcal L(x) \qquad \mbox{and} \qquad L(x) = \max\mathcal L(x).
\]
(In our setting, it is easy to argue that $\mathcal L(x)$ must be finite, so the maximum and minimum above are both well defined.)  To give the reader a feel for these invariants, in Table 2 we list all the McNugget numbers from 1 to 50 and their associated values $\mathcal{L}(x)$, $\ell(x)$, and $L(x)$.

\begin{table}[t]\label{tb:mcnuggetlengthsets}
\small{
\[
\begin{array}{||c|c|c|c||c|c|c|c||c|c|c|c||}
\hline
x & \mathcal{L}(x) & \ell(x) & L(x) & x & \mathcal{L}(x) & \ell(x) & L(x) & x & \mathcal{L}(x) & \ell(x) & L(x) \\
\hline
\mathbf{0} & \{0\} & 0 & 0 &
\mathbf{27} & \{3,4\} & 3 & 4 &
\mathbf{41} & \{4\} & 4 & 4 \\
\hline
\mathbf{6} & \{1\} & 1 & 1 &
\mathbf{29} & \{2\} & 2 & 2 &
\mathbf{42} & \{5,6,7\} & 5 & 7 \\
\hline
\mathbf{9} & \{1\} & 1 & 1 &
\mathbf{30} & \{4,5\} & 4 & 5 &
\mathbf{44} & \{4,5\} & 4 & 5 \\
\hline
\mathbf{12} & \{2\} & 2 & 2 &
\mathbf{32} & \{3\} & 3 & 3 &
\mathbf{45} & \{5,6,7\} & 5 & 7 \\
\hline
\mathbf{15} & \{2\} & 2 & 2 &
\mathbf{33} & \{4,5\} & 4 & 5 &
\mathbf{46} & \{3\} & 3 & 3 \\
\hline
\mathbf{18} & \{2,3\} & 2 & 3 &
\mathbf{35} & \{3\} & 3 & 3 &
\mathbf{47} & \{4,5\} & 4 & 5 \\
\hline 
\mathbf{20} & \{1\} & 1 & 1 &
\mathbf{36} & \{4,5,6\} & 4 & 6 &
\mathbf{48} & \{6,7,8\} & 6 & 8 \\
\hline
\mathbf{21} & \{3\} & 3 & 3 &
\mathbf{38} & \{3,4\} & 3 & 4 &
\mathbf{49} & \{3\} & 3 & 3 \\
\hline
\mathbf{24} & \{3,4\} & 3 & 4 &
\mathbf{39} & \{5,6\} & 5 & 6 &
\mathbf{50} & \{5,6\} & 5 & 6 \\
\hline 
\mathbf{26} & \{2\} & 2 & 2 &
\mathbf{40} & \{2\} & 2 & 2 & 
& & & \\
\hline
\end{array}
\]}
\caption{The McNugget numbers from 0 to 50 with $\mathcal{L}(x)$, $\ell(x)$, and $L(x)$.}
\end{table}

The following recent result describes the functions $L(x)$ and $\ell(x)$ for elements $x \in \nm$ that are \textit{sufficiently large} with respect to the generators.  Intuitively, Theorem~\ref{t:maxminlenquasi} says that for ``most'' elements $x$, any factorization with maximal length is almost entirely comprised of $n_1$, so $L(x + n_1)$ is obtained by taking a maximum length factorization for $x$ and adding one additional copy of $n_1$.  In general, the ``sufficiently large'' hypothesis is needed, since, for example, both $41 = 2 \cdot 9 + 1 \cdot 23$ and $50 = 5 \cdot 10$ are maximum length factorizations in the numerical monoid $\langle 9, 10, 23 \rangle$.  

\begin{thm}[{\cite[Theorems~4.2 and~4.3]{BOP1}}]\label{t:maxminlenquasi}
Suppose $\nm$ is a numerical monoid.  If $x > n_1n_k$, then
\[
L(x + n_1) = L(x) + 1, 
\]
and if $x > n_{k-1}n_k$, then
\[
\ell(x + n_k) = \ell(x) + 1.
\]
\end{thm}

We will return to this result in Section~\ref{sec:lencalculations}, where we give a closed formula for $L(x)$ and $\ell(x)$ that holds for all $x \in \cmm$.  

Given our definitions to this point, we can now mention perhaps the most heavily studied invariant in the theory of non-unique factorizations.  For $x \in \nm$, the ratio 
\[
\rho(x) = \frac{L(x)}{\ell(x)},
\]
is called the \textit{elasticity} of $x$, and 
\[
\rho(\nm) = \sup\{\rho(x) \mid x\in \nm\}
\]
is the \textit{elasticity} of $\nm$.  The elasticity of an element $n \in \nm$ measures how ``spread out'' its factorization lengths are; the larger $\rho(n)$ is, the more spread out $\mathcal L(n)$~is.  To this end, the elasticity $\rho(\nm)$ encodes the highest such ``spread'' throughout the entire monoid.  For example, if $\rho(\nm) = 2$, then the maximum factorization length of any element $n \in \nm$ is at most twice its minimum factorization length.  

A formula for the elasticity of a general numerical monoid, given below, was given in \cite{CHM}, and was the result of an undergraduate research project.  

\begin{thm}[\cite{CHM}, Theorem~2.1 and Corollary~2.3]\label{t:generalelasticity}
The elasticity of the numerical monoid $\nm$ is 
\[
\rho(\nm) = \frac{n_k}{n_1}.
\]
Moreover, $\rho(n) = \frac{n_k}{n_1}$ precisely when $n$ is an integer multiple of the least common multiple of $n_1$ and $n_k$, and for any rational $r < \frac{n_k}{n_1}$, there are only finitely many elements $x \in \nm$ with $\rho(x)\leq r$.  
\end{thm}

The significance of the final statement in Theorem~\ref{t:generalelasticity} is that there are rationals $1 \leq q \leq \frac{n_k}{n_1}$ that do not lie in the set $\{\rho(x) \mid x \in \nm\}$ and hence $\{\rho(x) \mid x \in \nm\} \subsetneq \mathbb Q \cap [1, \frac{n_k}{n_1}]$ (to use terminology from the literature, numerical monoids are not \emph{fully elastic}).  Figure~\ref{fig:mcnuggetelasticity} depicts the elasticities of elements of $\cmm$ up to $n = 400$; indeed, as $n$ increases, the elasticity $\rho(n)$ appears to converge to $\frac{10}{3} = \rho(\cmm)$.  In general, the complete image $\{\rho(x) \mid x \in \nm\}$ has been determined by Barron, O'Neill, and Pelayo in another student co-authored paper \cite[Corollary 4.5]{BOP1}; we direct the reader there for a thorough mathematical description of Figure~\ref{fig:mcnuggetelasticity}.  

\begin{figure}[tbp]
\begin{center}
\includegraphics[width=5in]{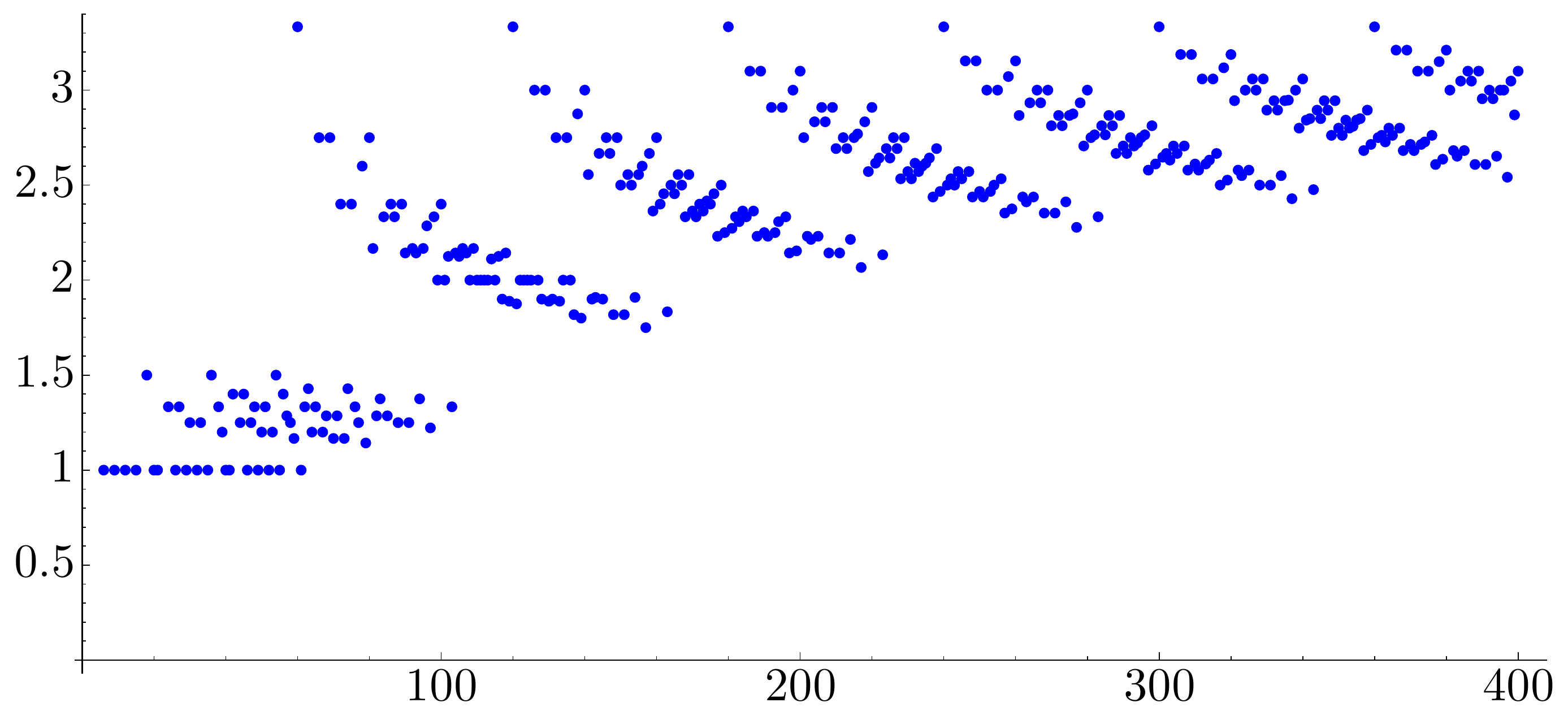}
\end{center}
\caption[A plot of the elasticities in the McNugget monoid]{A plot depicting the elasticity function $\rho(n)$ for $n \in \cmm$.}
\label{fig:mcnuggetelasticity}
\end{figure}

We close our discussion of elasticity with the following.

\begin{cor}\label{c:mcnuggetelasticity}
The elasticity of the McNugget monoid is
\[
\rho(\cmm) = \frac{10}{3}.
\]
\end{cor}

While a popular invariant to study, the elasticity only tells us about the largest and smallest elements of $\mathcal{L}(x)$.  Looking at Table 2, it appears that the length sets of the first few McNugget numbers are uniformly constructed (each is of the form $[a,b]\cap \mathbb N$ for positive integers $a$ and $b$).  One need not look too much further to break this pattern; the element $60 \in \cmm$ has
\[
\mathsf Z(60) = \{(0,0,3), (1,6,0), (4,4,0), (7,2,0), (10,0,0)\}
\]
and thus
\[
\mathcal L(60) = \{3, 7, 8, 9, 10\}.
\]
This behavior motivates the following ``finer'' factorization invariant.  Fix $x \in \nm$, and let $\mathcal L(x) = \{m_1, \ldots, m_t\}$ with $m_1 < m_2 < \cdots < m_t$.
Define the \emph{delta set} of $x$ as
\[
\Delta(x) = \{m_i - m_{i-1} \mid 2 \leq i \leq t\},
\]
and the \emph{delta set} of $\nm$ as
\[
\Delta(\nm) = \bigcup_{x \in \nm} \Delta(x).
\]
The study of the delta sets of numerical monoids (and more generally, of cancellative commutative monoids) has been an extremely popular topic; many such papers feature results from REU programs (see, for instance, \cite{BCKR,CDHK,CGLMS,CHK,CKLNZ,CK}).

From Table~\ref{tb:mcnuggetexpansions} we see that the McNugget numbers from 1 to 50 all have delta set $\emptyset$ or $\{1\}$, and we have further showed that $\Delta(60) = \{1,4\}$.  What is the delta set of $\cmm$ and moreover, what possible subsets of this set occur as $\Delta(x)$ for some $x \in \cmm$?  We will address those questions is Section~\ref{sec:deltasetcalculations}, with the help of a result from \cite{CHK}, stated below as Theorem~\ref{t:deltaperiodic}.  

One of the primary difficulties in determining the set $\Delta(\nm)$ is that even though each element's delta set $\Delta(x)$ is finite, the definition of $\Delta(\nm)$ involves the union of infinitely many such sets.  The key turns out to be a description of the sequence $\{\Delta(x)\}_{x \in \nm}$ for large $x$ (note that this is a sequence of sets, not integers).  Baginski conjectured during the writing of~\cite{BCKR} that this sequence is eventually periodic, and three years later this was settled in the affirmative, again in an REU project.  

\begin{thm}[{\cite[Theorem~1 and Corollary~3]{CHK}}]\label{t:deltaperiodic}
For $x \in \nm$, 
\[
\Delta(x) = \Delta(x + n_1n_k)
\]
whenever $x > 2kn_2n_k^2$.  In particular, 
\[
\Delta(\nm) = \bigcup_{x \in D} \Delta(x)
\]
where $D = \{x \in \nm \mid x \le 2kn_2n_k^2 + n_1n_k\}$ is a finite set.
\end{thm}

Thus $\Delta(\nm)$ can be computed in finite time.  The bound given in Theorem~\ref{t:deltaperiodic} is far from optimal; it is drastically improved in \cite{GMV}, albeit with a much less concise formula.  For convenience, we will use the bound given above in our computation of $\Delta(\cmm)$ in Section~\ref{sec:deltasetcalculations}.

\subsection{Beyond the length set}
\label{sec:beyondlengthset}

We remarked earlier that no element of a numerical monoid is prime.  Let's consider this more closely in $\cmm$.  For instance, since 6 is not prime, there is a sum $x + y$ in $\cmm$ such that 6 divides $x + y$, but 6 does not divide $x$ nor does 6 divide $y$ (take, for instance, $x = y = 9$).  But note that 6 satisfies the following slightly weaker property.  Suppose that 6 divides a sum $x_1+ \cdots + x_t$ where $t > 3$.  Then there is a subsum of at most 3 of the $x_i$'s that 6 does divide.   To see this, notice that if 6 divides any of the $x_i$'s, then we are done.  So suppose it does not.  If 9 divides both $x_i$ and $x_j$, then 6 divides $x_i + x_j$ since 6 divides $9 + 9$.  If no two $x_i$'s are divisible by 9, then at least 3 $x_i$'s are divisible by 20, and nearly identical reasoning to the previous case completes the argument.  This value of 3 offers some measure as to how far 6 is from being prime, and motivates the following definition.  

\begin{definition}\label{d:omega}
Let $\nm$ be a numerical monoid.  For any nonzero $x \in \nm$, define $\omega(x) = m$ if $m$ is the smallest positive integer such that whenever $x$ divides $x_1 + \cdots + x_t$, with $x_i \in \nm$, then there is a set $T \subset \{1,2,\ldots,t\}$ of indices with $|T| \leq m$ such that $x$ divides $\sum_{i \in T} x_i$.  
\end{definition}  

Using Definition~\ref{d:omega}, a prime element would have $\omega$-value 1, so $\omega(x)$ can be interpreted as a measure of how far $x$ is from being prime.  In $\cmm$, we argued that $\omega(6) = 3$; a similar argument yields $\omega(9) = 3$ and $\omega(20) = 10$.  Notice that the computation of $\omega(x)$ is dependent more on $\mathsf Z(x)$ than $\mathcal L(x)$, and hence encodes much different information than either $\rho(x)$ or $\Delta(x)$.   

Let us more closely examine the argument that $\omega(6) = 3$.  The key is that 6 divides $9 + 9$ and $20 + 20 + 20$, but does not divide any subsum of either.  Indeed, the latter of these expressions yields a lower bound of $\omega(6) \ge 3$, and the given argument implies that equality holds.  With this in mind, we give the following equivalent form of Definition~\ref{d:omega}, which often simplifies the computation of $\omega(x)$.  
 
\begin{thm}[{\cite[Proposition~2.10]{OP2}}]\label{t:omegaequiv}
Suppose $\nm$ is a numerical monoid and $x \in \nm$.  The following conditions are equivalent.
\begin{enumerate}
\item[(a)] $\omega(x) = m$.
\item[(b)] $m$ is the maximum length of a sum $x_1 + \cdots + x_t$ of irreducible elements in $\nm$ with the property that (i) $x$ divides $x_1 + \cdots + x_t$, and (ii) $x$ does not divide $x_1 + \cdots + x_{j-1} + x_{j+1} + \cdots + x_t$ for $1 \le j \le t$.  
\end{enumerate}
\end{thm}

\noindent
The sum $x_1 + \cdots + x_t$ alluded to in part (b) above is called a \textit{bullet} for~$x$.  Hence, $20 + 20 + 20$ is a bullet for $6$ in $\cmm$, and moreover has maximal length.  The benefit of Theorem~\ref{t:omegaequiv} is twofold: (i) each $x \in \nm$ has only finitely many bullets, and (ii) the list of bullets can be computed in a similar fashion to the set $\mathsf Z(x)$ of factorizations.  We refer the reader to \cite{ACKT,BOP2}, both of which give explicit algorithms (again resulting from undergraduate research projects) for computing $\omega$-values.  

Our goal is to completely describe the behavior of the $\omega$-function of the McNugget Monoid.  We do so in Section~\ref{sec:omegacalculations}, using the following result, which is clearly similar in spirit to Theorems \ref{t:maxminlenquasi} and \ref{t:deltaperiodic}.  


\begin{thm}[{\cite[Theorem~3.6]{OP1}}]\label{t:omegaquasi}
For $x \in \nm$ sufficiently large,
\[
\omega(x + n_1) = \omega(x) + 1.
\]
In particular, this holds for 
\[
x > \frac{F + n_2}{n_2/n_1 - 1}
\]
where $F = F(\nm)$ is the Frobenius number.  
\end{thm}

The similarity between Theorems~\ref{t:maxminlenquasi} and~\ref{t:omegaquasi} is not a coincidence.  While $L(x)$ and $\omega(x)$ are indeed different functions (for instance, $L(6) = 1$ while $\omega(6) = 3$), they are closely related; the $\omega$-function can be expressed in terms of max factorization length that is computed when some collections of generators are omitted.  We direct the interested reader to \cite[Section~6]{BOP2}, where an explicit formula of this form for $\omega(n)$ is given.

\section{Calculations for the Chicken McNugget monoid}
\label{sec:mcnuggetcalculations}

In the final section of this paper, we give explicit expressions for $L(x)$, $\ell(x)$, $\Delta(x)$ and $\omega(x)$ for every $x \in \cmm$.  The derivation of each such expression makes use of a theoretical result in Section~\ref{sec:toolkit}.  

We note that each of the formulas provided in this section could also be derived in a purely computational manner, using Theorems~\ref{t:maxminlenquasi}, \ref{t:deltaperiodic}, and~\ref{t:omegaquasi} and the inductive algorithms introduced in \cite{BOP2} (indeed, these computations finish in a reasonably short amount of time using the implementation in the \texttt{numericalsgps} package discussed in Section~\ref{sec:computers}).  However, several of the following results identify an interesting phenomenon that distinguish $\cmm$ from more general numerical monoids (see the discussion preceeding Question~\ref{q:nodissonance}), and the arguments that follow give the reader an idea of how theorems involving factorization in numerical monoids can be proven.  

\subsection{Calculating factorization lengths}
\label{sec:lencalculations}

Theorem~\ref{t:maxminlenquasi} states that $L(x + n_1) = L(x) + 1$ and $\ell(x + n_k) = \ell(x) + 1$ for sufficiently large $x \in \nm$.  but, it was observed during the writing of~\cite{BOP1} that for many numerical monoids, the ``sufficiently large'' requirement is unecessary.  As it turns out, one such example is the McNugget monoid~$\cmm$, which we detail below.  

\begin{thm}\label{t:mcnuggetmax}
For each $x \in \cmm$, $L(x + 6) = L(x) + 1$.  In particular, if we write $x = 6q + r$ for $q, r \in \mathbb N$ and $r < 6$, then
\[
L(x) = \left\{\begin{array}{l@{\qquad}l}
q & \textnormal{if } r = 0 \textnormal{ or } 3, \\
q - 5 & \textnormal{if } r = 1, \\
q - 2 & \textnormal{if } r = 2 \textnormal{ or } 5, \\
q - 4 & \textnormal{if } r = 4, \\
\end{array}\right.
\]
for each $x \in \cmm$.  
\end{thm}

\begin{proof}
Fix $x \in \cmm$ and a factorization $(a, b, c)$ of $x$.  
If $b > 1$, then $x$ has another factorization $(a + 3, b - 2, c)$ with length $a + b + c + 1$.  Similarly, if $c \ge 3$, then $(a + 10, b, c - 3)$ is also a factorization of $x$ and has length $a + b + c + 7$.  This implies that if $(a, b, c)$ has maximum length among factorizations of $x$, then $b \le 1$ and $c \le 2$.  Upon inspecting Table~\ref{tb:mcnuggetexpansions}, we see that unless $x \in \{0, 9, 20, 29, 40, 49\}$, we must have $a > 0$.  

Now, assume $(a,b,c)$ has maximum length among factorizations of $x$.  We claim $(a+1,b,c)$ is a factorization of $x + 6$ with maximum length.  From Table~\ref{tb:mcnuggetexpansions}, we see that since $x \in \cmm$, we must have $x + 6 \notin \{0, 9, 20, 29, 40, 49\}$, meaning any maximum length factorization of $x+6$ must have the form $(a' + 1, b', c')$.  This yields a factorization $(a', b', c')$ of $x$, and since $(a,b,c)$ has maximum length, we have $a + b + c \ge a' + b' + c'$.  As such, $(a+1,b,c)$ is at least as long as $(a' + 1, b', c')$, and the claim is proved.  Thus, 
\[
L(x+6) = a + 1 + b + c = L(x) + 1.
\]

From here, the given formula for $L(x)$ now follows from the first claim and the values $L(0)$, $L(9)$, $L(20)$, $L(29)$, $L(40)$, and $L(49)$ in Table~\ref{tb:mcnuggetlengthsets}.  
\end{proof}

A similar expression can be obtained for $\ell(x)$, ableit with 20 cases instead of 6, this time based on the value of $x$ modulo 20.  We encourage the reader to adapt the argument above for Theorem~\ref{t:mcnuggetmin}.  

\begin{thm}\label{t:mcnuggetmin}
For each $x \in \cmm$, $\ell(x + 20) = \ell(x) + 1$.  In particular, if we write $x = 20q + r$ for $q, r \in \mathbb N$ and $r < 20$, then
\[
\ell(x) = \left\{\begin{array}{l@{\qquad}l}
q & \textnormal{if } r = 0, \\
q + 1 & \textnormal{if } r = 6, 9, \\
q + 2 & \textnormal{if } r = 1, 4, 7, 12, 15, 18, \\
q + 3 & \textnormal{if } r = 2, 5, 10, 13, 16, \\
q + 4 & \textnormal{if } r = 8, 11, 14, 19, \\
q + 5 & \textnormal{if } r = 3, 17, \\
\end{array}\right.
\]
for each $x \in \cmm$.  
\end{thm}

Theorems~\ref{t:mcnuggetmax} and~\ref{t:mcnuggetmin} together yield a closed form for $\rho(x)$ that holds for all $x \in \cmm$.  Since $\text{lcm}(6,20) = 60$ cases are required, we leave the construction of this closed
form to the interested reader.

\subsection{Calculating delta sets}
\label{sec:deltasetcalculations}

Unlike maximum and minimum factorization length, $\Delta(x)$ is periodic for sufficiently large $x \in \cmm$.  For example, a computer algebra system can be used to check that $\Delta(91) = \{1\}$ while $\Delta(211) = \{1,2\}$.  Theorem~\ref{t:deltaperiodic} guarantees $\Delta(x + 120) = \Delta(x)$ for $x > 21600$, but some considerable reductions can be made.  In~particular, we will reduce the period from 120 down to 20, and will show that equality holds for all $x \ge 92$ (that is to say, 91~is the largest value of $x$ for which $\Delta(x + 20) \ne \Delta(x)$).  

\begin{thm}\label{t:mcnuggetdelta}
Each $x \in \cmm$ with $x \ge 92$ has $\Delta(x + 20) = \Delta(x)$.  Moreover,
\[
\Delta(x) = \left\{\begin{array}{l@{\qquad}l}
\{1\} & \textnormal{if } r = 3, 8, 14, 17, \\
\{1, 2\} & \textnormal{if } r = 2, 5, 10, 11, 16, 19, \\
\{1, 3\} & \textnormal{if } r = 1, 4, 7, 12, 13, 18, \\
\{1, 4\} & \textnormal{if } r = 0, 6, 9, 15, \\
\end{array}\right.
\]
where $x = 20q + r$ for $q, r \in \mathbb N$ and $r < 20$.   Hence $\Delta(\cmm)=\{1,2,3,4\}$.
\end{thm}

\begin{proof}
We will show that $\Delta(x + 20) = \Delta(x)$ for each $x > 103$.  The remaining claims can be verified by extending Table 2 using computer software.  

Suppose $x > 103$, fix a factorization $(a,b,c)$ for $x$, and let $l = a + b + c$.  If $c \ge 3$, then $x$ also has factorizations $(a + 10, b, c - 3)$, $(a + 7, b + 2, c - 3)$, $(a + 4, b + 4, c - 3)$, and $(a + 1, b + 6, c - 3)$, meaning 
\[
\{l, l + 4, l + 5, l + 6, l + 7\} \subset \mathcal L(x).
\]
Alternatively, since $x > 103$, if $c \le 2$, then $6a + 9b \ge 63$, and thus
\[
l \ge a + b + 2 \ge 9 \ge \ell(x) + 4.
\]
The above arguments imply (i) any gap in successive lengths in $\mathcal L(x)$ occurs between $\ell(x)$ and $\ell(x) + 4$, and (ii) every factorization with length in that interval has at least one copy of 20.  As such, $x + 20$ has the same gaps between $\ell(x + 20)$ and $\ell(x + 20) + 4$ as $x$ does between $\ell(x)$ and $\ell(x) + 4$, which proves $\Delta(x + 20) = \Delta(x)$ for all $x > 103$.  
\end{proof}

With a slightly more refined argument than the one given above, one can prove without the use of software that $\Delta(x + 20) = \Delta(x)$ for all $x \ge 92$.  We~encourage the interested reader to work out such an argument.  

\subsection{Calculating \texorpdfstring{$\omega$}{omega}-primality}
\label{sec:omegacalculations}

We conclude our study of $\cmm$ with an expression for the $\omega$-primality of $x \in \cmm$ and show (in some sense) how far a McNugget number is from being prime.  We proceed in a similar fashion to Theorems~\ref{t:mcnuggetmax} and~\ref{t:mcnuggetmin}, showing that with only two exceptions, $\omega(x + n_1) = \omega(x) + 1$ for all $x \in \cmm$.  

\begin{thm}\label{t:mcnuggetomega}
With the exception of $x = 6$ and $x = 12$, every nonzero $x \in \cmm$ satisfies $\omega(x + 6) = \omega(x) + 1$.  In particular, we have
\[
\omega(x) = \left\{\begin{array}{l@{\qquad}l}
q & \textnormal{if } r = 0, \\
q + 5 & \textnormal{if } r = 1, \\
q + 7 & \textnormal{if } r = 2, \\
q + 2 & \textnormal{if } r = 3, \\
q + 4 & \textnormal{if } r = 4, \\
q + 9 & \textnormal{if } r = 5, \\
\end{array}\right.
\]
for each $x \ne 6, 12$, where $x = 6q + r$ for $q, r \in \mathbb N$ and $r < 6$.  
\end{thm}

\begin{proof}
Fix $x \in \cmm$.  Following the spirit of the proof of Theorem~\ref{t:mcnuggetmax}, we begin by proving each $x > 12$ has a maximum length bullet $(a,b,c)$ with $a > 0$.  Indeed, suppose $(0,b,c)$ is a bullet for $x$ for some $b, c \ge 0$.  The element $x \in \cmm$ also has some bullet of the form $(a',0,0)$, where $a'$ the smallest integer such that $6a' - x \in \cmm$.  Notice $a' \ge 3$ since $x > 12$.  We consider several cases.  

\begin{itemize}
\item 
If $c = 0$, then $9b - x \in \cmm$ but $9b - x - 9 \notin \cmm$.  If $b \le 3$, then $a' \ge b$.  Otherwise, either $9(b-1)$ or $9(b-2)$ is a multiple of 6, and since $9(b-2) - x \notin \cmm$ as well, we see $a' \ge \frac{3}{2}(b-2) + 1 \ge b$.

\item 
If $b = 0$, then there are two possibilities.  If $c \le 3$, then $a' \ge c$.  Otherwise, either $20(c-1)$, $20(c-2)$ or $20(c-3)$ is a multiple of 6, so we conclude $a' \ge \frac{10}{3}(c - 3) + 1 \ge c$.  

\item 
If $b, c > 0$, then $9b + 20c - x - 9, 9b + 20c - x - 20 \notin \cmm$, so $9b + 20c - x$ is either $0$, $6$, or $12$.  This means either $(3,b-1,c)$, $(2,b-1,c)$, or $(1,b-1,c)$ is also a bullet for $x$, respectively.  

\end{itemize}
In each case, we have constructed a bullet for $x$ at least as long as $(0,b,c)$, but with positive first coordinate, so we conclude $x$ has a maximal bullet with nonzero first coordinate.  

Now, using a similar argument to that given in the proof of Theorem~\ref{t:mcnuggetmax}, if $(a + 1, b, c)$ is a maximum length bullet for $x + 6$, then $(a, b, c)$ is a maximum length bullet for $x$.  This implies $\omega(x + 6) = \omega(x) + 1$ whenever $x + 6$ has a maximum length bullet with positive first coordinate, which by the above argument holds whenever $x > 12$.  This proves the first claim.  

The formula for $\omega(x)$ now follows from the first claim, the computations $\omega(9) = 3$ and $\omega(20) = 10$ from Section~\ref{sec:beyondlengthset}, and analogous computations for $\omega(15) = 4$, $\omega(18) = 3$, $\omega(29) = 13$, $\omega(40) = 10$, and $\omega(49) = 13$.  
\end{proof}

Figure~\ref{fig:mcnuggetomega} plots $\omega$-values of elements of the McNugget monoid $\cmm$.  Since $\omega(x + 6) = \omega(x) + 1$ for large $x \in \cmm$, most of the plotted points occur on one of 6 lines with slope $\frac{1}{6}$.  It is also evident in the plot that $x = 6$ and $x = 12$ are the only exceptions.  

\begin{figure}[tbp]
\begin{center}
\includegraphics[width=5in]{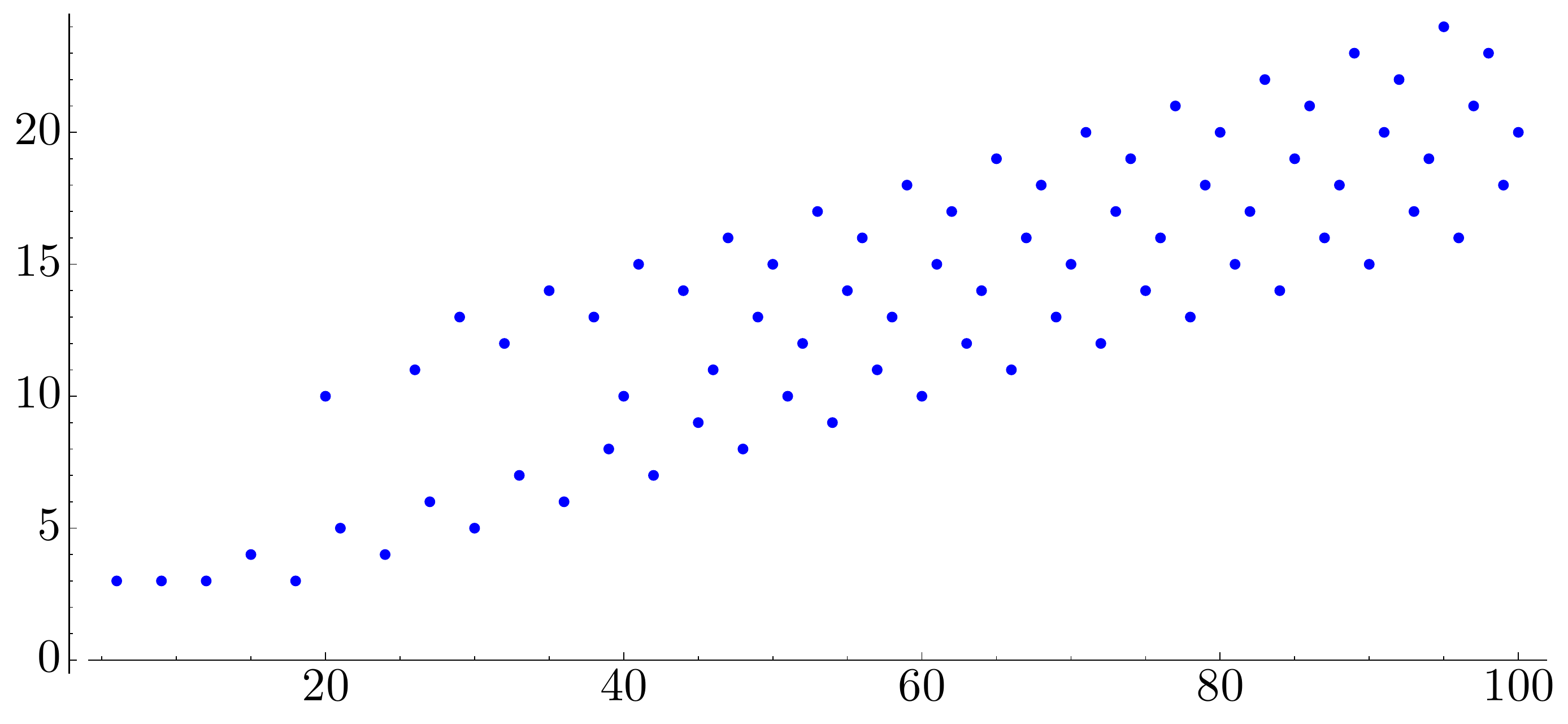}
\end{center}
\caption[A plot of $\omega$-primality in the McNugget monoid]{A plot depicting the $\omega$-primality function $\omega(n)$ for $n \in \cmm$.}
\label{fig:mcnuggetomega}
\end{figure}

Although $\omega(x + 6) = \omega(x) + 1$ does not hold for every $x \in \cmm$, there are some numerical monoids for which the ``sufficiently large'' hypothesis in Theorem~\ref{t:omegaquasi} can be dropped (for instance, any numerical monoids with 2 minimal generators has this property).  Hence, we conclude with a problem suitable for attack by undergraduates.

\begin{question}\label{q:nodissonance}
Determine which numerical monoids $\nm$ satisfy each of the following conditions for all $x$ (i.e., not just sufficiently large~$x$):
\begin{enumerate}
\item $L(x + n_1) = L(x) + 1$, 
\item $\ell(x + n_k) = \ell(x) + 1$, or 
\item $\omega(x + n_1) = \omega(x) + 1$. 
\end{enumerate}
\end{question}

\section{Appendix:\ computer software for numerical monoids}
\label{sec:computers}

Many of the computations referenced in this paper can be performed using the \texttt{numericalsgps} package \cite{DGM} for the computer algebra system \texttt{GAP}.  The~brief snippet of sample code below demonstrates how the package is used to compute various quantities discussed in this paper.  

\begin{verbatim}
gap> LoadPackage("num");
true
gap> McN:=NumericalSemigroup(6,9,20);
<Numerical semigroup with 3 generators>
gap> FrobeniusNumberOfNumericalSemigroup(McN);
43
gap> 43 in McN; 
false
gap> 44 in McN;
true
gap> FactorizationsElementWRTNumericalSemigroup(18,McN);
[ [ 3, 0, 0 ], [ 0, 2, 0 ] ]
gap> OmegaPrimalityOfElementInNumericalSemigroup(6,McN);
3
\end{verbatim}

This only scratches the surface of the extensive functionality offered by the \texttt{numericalsgps} package.  We encourage the interested reader to install and experiment with the package; instructions can be found on the official webpage, whose URL is included below.  

\begin{center}
\url{https://www.gap-system.org/Packages/numericalsgps.html}
\end{center}


\end{document}